\documentclass{amsart}
\usepackage[table]{xcolor}
\usepackage{tikz}
\usepackage{fullpage}
\usepackage{url}
\usepackage{float}
\usepackage{mathtools}

\usepackage[colorlinks=true, citecolor=blue]{hyperref}
\usepackage{amsmath,amssymb,amsfonts,amsthm}
\usepackage{graphicx}

\title{\textbf{Decoupling and Discrete Strichartz Estimates for Dispersive Equations on the Torus}}
\author{Yuda Chen}
\email{math.yuda.chen@gmail.com}
\address{School of Mathematical Sciences and LPMC, Nankai University, Tianjin 300071, P.R. China}
\thanks{This work was conceived and written while the author was visiting the Department of Mathematics at UC Berkeley.}
\date{\today}
\everymath{\displaystyle}

\newtheorem{theorem}{Theorem}[section]
\newtheorem{proposition}{Proposition}[section]

\newtheorem{conjecture}[proposition]{Conjecture}
\theoremstyle{remark}
\newtheorem{remark}[proposition]{Remark}

\theoremstyle{definition}

\begin{document}

\begin{abstract}
	This paper proves weighted discrete Strichartz, or equivalently discrete restriction, estimates for two-dimensional exponential sums with fractional and perturbed polynomial phases.  The new results cover three related classes of curves: the fractional Schr\"odinger-type curve $(n,n^{1+\nu})$, the fully fractional curve $(n^{1+\mu},n^{1+\nu})$ with $0<\mu<\nu$, and bounded perturbations of the polynomial curve $(n,n^k)$.  In each case we obtain $\ell^2$-weighted $L^p$ estimates with decoupling-type dependence on the frequency scale, including both low- and high-$p$ regimes.  The proofs are based on finite-type decoupling, together with localization and rescaling arguments adapted to curvature degeneracy, anisotropic homogeneity, and bounded perturbations of the phase.  The introduction frames these estimates from the viewpoint of decoupling theory, with arithmetic results obtained by efficient congruencing used only as comparison points for the known cases.
\end{abstract}

	\maketitle
	\section{Introduction}
	We begin by reviewing the classical discrete Strichartz estimates, which arise naturally in the study of periodic solutions to dispersive partial differential equations. Throughout this paper, we identify the torus $\mathbb{T}$ with $\mathbb{R}/\mathbb{Z}$. A quintessential example is the initial value problem for the linear Schr\"odinger equation on the space-time torus $\mathbb{T} \times \mathbb{T}$:
	
	$$i\partial_t u(x,t) - \frac{1}{2\pi}\partial_x^2 u(x,t) = 0, \quad (x,t) \in \mathbb{T} \times \mathbb{T},$$
	
	$$u(x,0) = u_0(x), \quad x \in \mathbb{T}.$$
	
	By taking the Fourier transform in the spatial variable $x$, the $n$-th spatial Fourier coefficient $\widehat{u}(n,t)$ of the solution satisfies the ordinary differential equation:
	
	$$i\partial_t \widehat{u}(n,t) + 2\pi n^2 \widehat{u}(n,t) = 0.$$
	
	Solving this yields $\widehat{u}(n,t) = a_n e^{2\pi i n^2 t}$, where $a_n = \widehat{u}(n,0)$ is the $n$-th Fourier coefficient of the initial data $u_0(x)$. Consequently, the full solution to the periodic Schr\"odinger equation can be expressed as the exponential sum:
	
	$$u(x,t) = \sum_{n \in \mathbb{Z}} a_n e^{2\pi i(nx+n^2t)}.$$
	
	In practice, we are often interested in frequency-localized initial data. Suppose $u_0(x) = \sum_{n=0}^N a_n e^{2\pi i n x}$ for some integer $N \ge 1$. By Plancherel's theorem, the $L^2$ norm of the initial data is given by:
	
	$$\|u_0\|_{L^2(\mathbb{T})} = \left( \sum_{n=0}^N |a_n|^2 \right)^{\frac12}.$$
	
	A central problem in partial differential equations and harmonic analysis is understanding the $L^p$ integrability of the solution $u(x,t)$ in terms of the $L^2$ norm of its initial data. Equivalently, one is asking for restriction-type estimates for exponential sums supported on a discrete curve. For the Schr\"odinger equation this curve is the discrete parabola
	\[
		n\mapsto(n,n^2).
	\]
	This is elegantly answered by the foundational discrete Strichartz estimate, famously established by Bourgain's groundbreaking paper~\cite{Bourgain93a} in 1993 using arithmetic methods.
	
	\begin{proposition}[Bourgain~\cite{Bourgain93a}]\label{schrodinger}
		
		Let $N \ge 1$ be an integer and $\varepsilon > 0$. Suppose $u_0 \in L^2(\mathbb{T})$ is frequency-localized such that $u_0(x) = \sum_{n=0}^N a_n e^{2\pi i n x}$. Then for every $p\ge 2$, the corresponding solution $u(x,t)$ to the periodic linear Schr\"odinger equation satisfies the spacetime bound:
		
		$$\|u\|_{L^p(\mathbb{T}^2)} \le C_{p,\varepsilon} N^\varepsilon \left(1 + N^{\frac{1}{2}-\frac{3}{p}}\right) \|u_0\|_{L^2(\mathbb{T})},$$
		
		where $C_{p,\varepsilon}$ is a positive constant depending only on $p$ and $\varepsilon$.
		
		Equivalently, written purely in terms of the Fourier coefficients, one has the classical exponential sum estimate:
		
		$$\left\| \sum_{n=0}^N a_n e^{2\pi i(nx+n^2t)} \right\|_{L^p([0,1]^2)} \lesssim_{p,\varepsilon} N^\varepsilon \left(1 + N^{\frac{1}{2}-\frac{3}{p}}\right) \left( \sum_{n=0}^N |a_n|^2 \right)^{\frac12}.$$
		
	\end{proposition}
	If we consider the KdV equation, there is also a corresponding conjecture of discrete Strichartz estimate.
	\begin{conjecture}\label{kdvconjecture}
		
		Let $N \ge 1$ be an integer and $\varepsilon > 0$. Suppose $u_0 \in L^2(\mathbb{T})$ is frequency-localized such that $u_0(x) = \sum_{n=0}^N a_n e^{2\pi i n x}$. Then for every $p \ge 2$, the corresponding solution $u(x,t)$ to the periodic linear KdV equation satisfies the spacetime bound:
		
		$$\|u\|_{L^p(\mathbb{T}^2)} \le C_{p,\varepsilon} N^\varepsilon \left(1 + N^{\frac{1}{2}-\frac{4}{p}}\right) \|u_0\|_{L^2(\mathbb{T})},$$
		
		where $C_{p,\varepsilon}$ is a positive constant depending only on $p$ and $\varepsilon$.
		
		Equivalently, written purely in terms of the Fourier coefficients, one has the classical exponential sum estimate:
		
		$$\left\| \sum_{n=0}^N a_n e^{2\pi i(nx+n^3t)} \right\|_{L^p([0,1]^2)} \lesssim_{p,\varepsilon} N^\varepsilon \left(1 + N^{\frac{1}{2}-\frac{4}{p}}\right) \left( \sum_{n=0}^N |a_n|^2 \right)^{\frac12}.$$
		
	\end{conjecture}
	In 1993, using purely arithmetic methods, Bourgain~\cite{Bourgain93b} established Conjecture~\ref{kdvconjecture} for the cases $2\le p\le 6$. In 2013, Hu and Li~\cite{HL13} proved the cases $p\ge14$. Lai and Ding~\cite{LD18} improved it to $p\ge12$ in 2018. Later, Hughes and Wooley~\cite{HW22} used efficient congruencing methods to extend the weighted $\ell^2$ range to $p\ge10$ in 2022. There is also an important unweighted ninth-moment result: Wooley~\cite{Wooley15} proved an essentially optimal estimate for the ninth moment of the cubic Weyl sum with phase $\alpha n^3+\beta n$. In the present notation this gives
	\[
		\left\|\sum_{n=1}^N e^{2\pi i(nx+n^3t)}\right\|_{L^9([0,1]^2)}
		\lesssim_{\varepsilon} N^{5/9+\varepsilon},
	\]
	or, after comparison with $\|1\|_{\ell^2([1,N])}\sim N^{1/2}$, the exponent $E_9=1/18$. This $L^9$ result is for the constant coefficient case $a_n\equiv1$; it should not be confused with the weighted $\ell^2$ restriction estimate in Conjecture~\ref{kdvconjecture}.
	
	During the 2010s, decoupling theory underwent rapid development. Bourgain and Demeter~\cite{BD15} utilized the decoupling method to provide a new proof of Proposition~\ref{schrodinger}. For standard background on decoupling and its applications to restriction theory, see Demeter's book~\cite{Demeter20} and Guo's book~\cite{Guo26}. However, this method cannot be directly transplanted to the proof of Conjecture~\ref{kdvconjecture} in the expected range. Fortunately, in 2020, Biswas et al.~\cite{BGLSX20} introduced an improvement to the decoupling inequality for curves with vanishing curvature. We can leverage this idea to complete the decoupling-level proof of Conjecture~\ref{kdvconjecture} for the case $2\le p\le6$, but since this result is not one of the main contributions of the present paper, we record it only as a remark.
	
	\begin{remark}\label{kdv}
		Let $N \ge 1$ be an integer and $\varepsilon > 0$. Suppose $u_0 \in L^2(\mathbb{T})$ is frequency-localized such that $u_0(x) = \sum_{n=0}^N a_n e^{2\pi i n x}$. Then for every $p \ge 2$, the corresponding solution $u(x,t)$ to the periodic linear KdV equation satisfies the spacetime bound:
		
		$$\|u\|_{L^p(\mathbb{T}^2)} \le C_{p,\varepsilon} N^\varepsilon \left(1 + N^{\frac{1}{2}-\frac{3}{p}}\right) \|u_0\|_{L^2(\mathbb{T})},$$
		
		where $C_{p,\varepsilon}$ is a positive constant depending only on $p$ and $\varepsilon$.
		
		Equivalently, written purely in terms of the Fourier coefficients, one has the exponential sum estimate:
		
		$$\left\| \sum_{n=0}^N a_n e^{2\pi i(nx+n^3t)} \right\|_{L^p([0,1]^2)} \lesssim_{p,\varepsilon} N^\varepsilon \left(1 + N^{\frac{1}{2}-\frac{3}{p}}\right) \left( \sum_{n=0}^N |a_n|^2 \right)^{\frac12}.$$
		
		This estimate is sharp at the $L^6$ decoupling level, but it is weaker than Conjecture~\ref{kdvconjecture} for $p>6$.
		
		The proof of Theorem~\ref{singlefraction} also applies to this proposition.
	\end{remark}
	
	Building upon these classical frameworks, it is a natural mathematical progression to extend discrete Strichartz estimates to more general phase functions. We first consider the case where the spatial and temporal frequencies are dictated by polynomials $P$ and $Q$ with integer coefficients. In this setting, the most robust formulation is the exponential sum estimate associated with the polynomial curve
	\[
	n\mapsto (P(n),Q(n)).
	\]
	Since both $P(n)$ and $Q(n)$ take integer values for $n\in\mathbb Z$, the corresponding exponential sum is naturally defined on the space-time torus $[0,1]^2$.
	
	One may interpret this estimate as a periodic dispersive estimate whenever the polynomial parametrization genuinely determines a single spatial Fourier multiplier, for instance under the additional condition that $P(n)=P(m)$ implies $Q(n)=Q(m)$ on the relevant frequency set. In general, however, $P$ need not be injective on integers, and therefore the evolution is not always determined solely by the initial datum $\sum a_n e^{2\pi iP(n)x}$. For this reason, we formulate the general polynomial result directly at the level of exponential sums.
	
	\begin{remark}[Polynomial phases]\label{polynomial}
		Let $P$ and $Q$ be polynomials with integer coefficients with respective degrees $k_{1}$ and $k_{2}$ such that $1\le k_{1}\le k_{2}$, and suppose that $P'$ and $Q'$ are linearly independent over $\mathbb{Q}$. For any integer $N \ge 1$, any $\varepsilon > 0$, and any finite sequence of complex coefficients $\{a_n\}_{n=1}^N$, the following estimate holds for every $p \ge 2$:
		$$
		\left\| \sum_{n=1}^N a_n e^{2\pi i(P(n)x+Q(n)t)} \right\|_{L^p([0,1]^2)}
		\lesssim_{P,Q,p,\varepsilon}
		N^\varepsilon
		\left(1 + N^{\frac{1}{2}-\frac{3}{p}}\right)
		\left( \sum_{n=1}^N |a_n|^2 \right)^{\frac12}.
		$$
		This is a standard finite-type decoupling consequence for polynomial curves. Since the present paper focuses on fractional and perturbed phases, we keep this result only as a comparison point; the decoupling reduction behind it is briefly recalled in Section~\ref{polynomialproof}.
	\end{remark}
	
	Beyond integer-coefficient polynomials, modern analysis has naturally generalized this theory to fractional dispersion. For a fixed real number $\nu > 0$, we consider the fractional Schr\"odinger-type equation on the spatial torus and on a unit time interval:
	
	$$i\partial_t u(x,t) + 2\pi \left(-\frac{1}{4\pi^2}\partial_x^2\right)^{\frac{1+\nu}{2}} u(x,t) = 0, \quad (x,t) \in \mathbb{T} \times [0,1].$$
	
	Given the standard frequency-localized initial data $u_0(x) = \sum_{n=1}^N a_n e^{2\pi i n x}$, the spatial frequency remains the integer $n$, guaranteeing spatial periodicity. The temporal evolution yields the single-fractional phase $nx+n^{1+\nu}t$. Unless $\nu$ is an integer in a special compatible situation, the temporal factor is not $1$-periodic in $t$, so the estimate is naturally understood locally on the time interval $[0,1]$. For this problem, we establish the discrete Strichartz estimate across the full continuous spectrum of $p$:
	
	\begin{theorem}\label{singlefraction}
		Let $\nu > 0$. For any integer $N \ge 1$ and any $\varepsilon > 0$, suppose $u_0 \in L^2(\mathbb{T})$ is frequency-localized such that $u_0(x) = \sum_{n=1}^N a_n e^{2\pi i n x}$. Then for every $p \ge 2$, the corresponding solution $u(x,t)$ to the fractional dispersive equation satisfies the local spacetime estimate:
		
		$$\|u\|_{L^p(\mathbb{T}\times[0,1])} \le C_{\nu,p,\varepsilon} N^{E_p + \varepsilon} \|u_0\|_{L^2(\mathbb{T})},$$
		
		where the loss exponent $E_p$ is determined continuously across two regimes. With $r = \max\{1+\nu,2\}$, we have
		
		$$E_p = \begin{cases} \frac{(r-1-\nu)(p-2)}{4p}, & \text{if } 2 \le p \le 6, \\ \frac{1}{2} + \frac{r-4-\nu}{p}, & \text{if } p > 6. \end{cases}$$
		
		Equivalently, in terms of the Fourier coefficients:
		
		$$\left\| \sum_{n=1}^N a_n e^{2\pi i (n x + n^{1+\nu} t)} \right\|_{L^p([0,1]^2)} \lesssim_{\nu,p,\varepsilon} N^{E_p + \varepsilon} \left( \sum_{n=1}^N |a_n|^2 \right)^{\frac12}.$$
	\end{theorem}
	
	The preceding result still uses the integer spatial frequency $n$. We next remove this structure and study a genuinely two-dimensional fractional curve. Let $0<\mu<\nu$ and consider
	\[
		n\mapsto(n^{1+\mu},n^{1+\nu}).
	\]
	This is no longer the graph of a periodic spatial multiplier on $\mathbb T$ in the usual sense, because the first coordinate is fractional. Nevertheless, the associated exponential sum is a natural model for discrete restriction on a fully fractional curve. This case is the main theorem of the paper. The proof requires an anisotropic version of the finite-type decoupling argument, since the two coordinates now have different fractional homogeneities.
	
	\begin{theorem}\label{doublefraction}
		Let $0 < \mu < \nu$. For any integer $N \ge 1$, any $\varepsilon > 0$, and any finite sequence of complex coefficients $\{a_n\}_{n=1}^N$, the following discrete spacetime estimate holds on the unit square $[0,1]^2$ for all $p \ge 2$:
		
		$$\left\| \sum_{n=1}^N a_n e^{2\pi i (n^{1+\mu}x + n^{1+\nu}t)} \right\|_{L^p([0,1]^2)} \lesssim_{\mu,\nu,p,\varepsilon} N^{E_p + \varepsilon} \left( \sum_{n=1}^N |a_n|^2 \right)^{\frac12},$$
		
		where the loss exponent $E_p$ is determined continuously across two regimes. With $r_1 = \max\{2, 1+\mu\}$ and $r_2 = \max\{2, 1+\nu\}$, we have
		
		$$E_p = \begin{cases} \frac{(p-2)(r_1 + r_2 - 2 - \mu - \nu)}{4p}, & \text{if } 2 \le p \le 6, \\ \frac{1}{2} + \frac{r_1 + r_2 - 5 - \mu - \nu}{p}, & \text{if } p > 6. \end{cases}$$
	\end{theorem}
	
	We finally examine the stability of the discrete Strichartz estimate under bounded perturbations of the dispersion relation. Let $k\ge2$ be an integer and let $f:\mathbb{Z}\to\mathbb{R}$ be a bounded function. For frequency-localized initial data supported on frequencies $\{0,1,\dots,N\}$, the associated local wave packet on $\mathbb{T}\times[0,1]$ takes the form
	
	$$u(x,t)=\sum_{n=0}^N a_n e^{2\pi i\left(nx+(n^k+f(n))t\right)}.$$
	
	The temporal factor is not necessarily $1$-periodic when $f(n)$ is not integer-valued, so the theorem is formulated directly as a local exponential sum estimate on $[0,1]^2$. The following theorem shows that the decoupling-level estimate obtained in the unperturbed monomial case remains valid, with the same dependence on $N$, for any bounded perturbation.
	
	\begin{theorem}\label{perturbation}
		For any integer $N \ge 1$, any integer $k\ge2$, any $\varepsilon > 0$, any sequence $\{\theta_n\}_{n=0}^N$ satisfying $\sup_{0\le n\le N} |\theta_n| \le C$, and any finite sequence of complex coefficients $\{a_n\}_{n=0}^N$, the following discrete spacetime estimate holds on the unit square $[0,1]^2$ for all $p \ge 2$:
		
		$$\left\|\sum_{n=0}^N a_n e^{2\pi i (n x + (n^k+\theta_{n})t)}\right\|_{L^p([0,1]^2)} \lesssim_{k,p,\varepsilon,C} N^\varepsilon \left(1+N^{\frac12-\frac3p}\right) \left(\sum_{n=0}^N |a_n|^2\right)^{\frac12}.$$
	\end{theorem}
	
	This result demonstrates that the discrete Strichartz estimate is robust under bounded modifications of the phase function. Even when the pure monomial $n^k$ is replaced by $n^k+\theta_n$ with $\{\theta_n\}$ bounded, the same growth exponent $\frac12-\frac3p$ (up to an $N^\varepsilon$ loss) persists. Hence the geometric structure used by decoupling is stable against perturbations that destroy the underlying polynomial arithmetic.
	
	Apart from certain specialized techniques, there are two primary approaches to problems involving discrete Strichartz estimates: decoupling and efficient congruencing. The former is the central methodology employed in this paper, while the latter is a number-theoretic approach pioneered by Wooley across a series of publications~\cite{Wooley12,Wooley13,Wooley16,Wooley17a,Wooley17b,Wooley19}. The table below presents known results and conjectured bounds for different exponents $E_p$ generated by different pairs $(g(n),h(n))$ in
	{\setlength{\fboxsep}{0.5pt}
	$$\left\| \sum_{n=1}^N a_n e^{2\pi i \left(\boxed{\scriptstyle g(n)}x + \boxed{\scriptstyle h(n)}t\right)} \right\|_{L^p([0,1]^2)} \lesssim_{p,\varepsilon} N^{\boxed{\scriptstyle E_{p}} + \varepsilon} \left( \sum_{n=1}^N |a_n|^2 \right)^{\frac12}.$$}
	
	{\begin{table}[H]
			\centering
		\renewcommand{\arraystretch}{1.35}
		\resizebox{\textwidth}{!}{\begin{tabular}{c|c|c|c|c}
		  $ (g(n),h(n)) $ & $ E_{p} $ by decoupling methods & $ E_{p} $ by efficient congruencing methods & $ E_{p} $ by other methods & Conjecture $ E_{p} $\\\hline
		  
		   $ (n,n^{2}) $ & \cellcolor{gray!30} Bourgain--Demeter~\cite{BD15} $ \left(\frac{1}{2}-\frac{3}{p}\right)_{+} $ & \cellcolor{gray!30} Wooley~\cite{Wooley19} $ \left(\frac{1}{2}-\frac{3}{p}\right)_{+} $ & \cellcolor{gray!30} Bourgain~\cite{Bourgain93a} $ \left(\frac{1}{2}-\frac{3}{p}\right)_{+} $ & $ \left(\frac{1}{2}-\frac{3}{p}\right)_{+} $ \\\hline
		   
		    $ (n,n^{3}) $ & Remark~\ref{kdv} $ \left(\frac{1}{2}-\frac{3}{p}\right)_{+} $ & \cellcolor{gray!30} \begin{tabular}{l}
		    	Hughes--Wooley~\cite{HW22} $ \left(\frac{1}{2}-\frac{4}{p}\right)_{+} $ for $ p\in[2,6]\cup[10,\infty) $\\
		    	Wooley~\cite{Wooley15} $ \left(\frac{1}{2}-\frac{4}{p}\right)_{+} $ for $ p\in[9,\infty) $ but $a_n\equiv1$ only
		    \end{tabular} & \begin{tabular}{l}
		    	Bourgain~\cite{Bourgain93b} $ \left(\frac{1}{2}-\frac{3}{p}\right)_{+} $\\
		    	Hu--Li~\cite{HL13} $ \left(\frac{1}{2}-\frac{4}{p}\right)_{+} $ for $ p\in[14,\infty) $\\
		    	Lai--Ding~\cite{LD18} $ \left(\frac{1}{2}-\frac{4}{p}\right)_{+} $ for $ p\in[12,\infty) $\\
		    	
		    \end{tabular} & $ \left(\frac{1}{2}-\frac{4}{p}\right)_{+} $ \\\hline
		    
		    $ \substack{(Poly_{k_{1}}(n),Poly_{k_{2}}(n))\\1\le k_{1}\le k_{2}} $ & Remark~\ref{polynomial} $ \left(\frac{1}{2}-\frac{3}{p}\right)_{+} $ &\cellcolor{gray!30} \begin{tabular}{l}Wooley~\cite{Wooley19} and Hughes--Wooley~\cite{HW22}\\
		    	$ \left(\frac12-\frac4p\right)_{+}$
		    	for $p\in[2,6]\cup[10,k_2(k_2+1))$\\
		    	$ \left(\frac12-\frac{k_1+k_2}{p}\right)_{+}$
		    	for $p\ge k_2(k_2+1)$ where $1\le k_1<k_2$ and $k_2\ge3$\end{tabular} & / & $ \left(\frac{1}{2}-\frac{k_{1}+k_{2}}{p}\right)_{+} $ \\\hline
		    
		    $(n,n^{1+\nu}) $ & \cellcolor{gray!30} Theorem~\ref{singlefraction} $ \begin{array}{l}\begin{cases} \frac{(r-1-\nu)(p-2)}{4p}, &  2 \le p \le 6, \\ \frac{1}{2} + \frac{r-4-\nu}{p}, &  p > 6. \end{cases}\\\text{where }r=\max\{2,1+\nu\}\end{array} $ &  /  & / & $ \left(\frac{1}{2}-\frac{2+\nu}{p}\right)_{+} $ \\\hline
		    
		    $\substack{(n^{1+\mu},n^{1+\nu})\\0<\mu<\nu} $ & \cellcolor{gray!30} Theorem~\ref{doublefraction} $ \begin{array}{l}\begin{cases} \frac{(p-2)(r_1 + r_2 - 2 - \mu - \nu)}{4p}, & 2 \le p \le 6, \\ \frac{1}{2} + \frac{r_1 + r_2 - 5 - \mu - \nu}{p}, &  p > 6. \end{cases}\\\text{where }r_1 = \max\{2, 1+\mu\},r_2 = \max\{2, 1+\nu\}\end{array} $ &  /  & / & $ \left(\frac{1}{2}-\frac{2+\mu+\nu}{p}\right)_{+} $ \\\hline
		    
		    $ (n,n^{k}+O(1)) $ & \cellcolor{gray!30} Theorem~\ref{perturbation} $ \left(\frac{1}{2}-\frac{3}{p}\right)_{+} $ & /  & / & $ \left(\frac{1}{2}-\frac{k+1}{p}\right)_{+} $ \\\hline
		\end{tabular}}\\
		Here $ (\cdot)_{+}=\max\{0,\cdot\}$, and $Poly_k$ denotes a polynomial with integer coefficients and degree $k$. Cells with \colorbox{gray!30}{gray background} indicate the best currently recorded bound within the corresponding row. 
		\label{table}
	\end{table}}
	
	The structure of this paper is as follows. Section~\ref{preliminaries} reviews the decoupling estimates used later. Section~\ref{proofsingle} proves Theorem~\ref{singlefraction}. Section~\ref{proofdouble} proves Theorem~\ref{doublefraction}. Section~\ref{proofperturbation} proves Theorem~\ref{perturbation} and gives two applications. Section~\ref{finalremarks} concludes with brief comments on the method.

\section{Preliminaries}\label{preliminaries}
	We will introduce the decoupling inequality for parabola first. We fix $\delta \in 2^{-\mathbb{N}}$, i.e. $\delta = 2^{-k}$ for some non-negative integer $k$. For an interval $I \subset [0,1]$ define the frequency region
	
	$$\Omega_I(\delta):=\left\{\left(\xi_1, \xi_1^2+\delta'\right): \xi_1 \in I,\; |\delta'| \le \delta^2\right\}.$$
	
	Let $p \in (1,\infty)$ and let $F: \mathbb{R}^2 \to \mathbb{C}$ be an $L^p$ function whose Fourier transform is supported on $\Omega_{[0,1]}(\delta)$. For each interval $I$ we denote by $F_I$ the Fourier restriction of $F$ to $\Omega_I(\delta)$, i.e.
	
	$$F_I := \mathcal{F}^{-1}\left(\widehat{F}\cdot \mathbf{1}_{\Omega_I(\delta)}\right).$$
	
	Because $1<p<\infty$, the $L^p$ boundedness of the Hilbert transform together with Fubini's theorem guarantees that each $F_I$ is again an $L^p(\mathbb{R}^2)$ function.
	
	Let $\mathcal{I}_\delta$ be the collection of all dyadic subintervals $I \subset [0,1]$ of length $\ell(I)=\delta$ (i.e. intervals of the form $[j\delta,(j+1)\delta)$ contained in $[0,1]$). The sets $\{\Omega_I(\delta)\}_{I\in\mathcal{I}_\delta}$ form a partition of $\Omega_{[0,1]}(\delta)$ (up to boundaries), and we have the decomposition
	
	$$F = \sum_{I\in\mathcal{I}_\delta} F_I.$$
	
	A Fourier decoupling inequality compares the $L^p$ norm of $F$ with the square root of the sum of squares of the $L^p$ norms of the pieces $F_I$:
	
	$$\left(\sum_{I\in\mathcal{I}_\delta} \|F_I\|_{L^p(\mathbb{R}^2)}^2\right)^{\frac12}.$$
	
	By the triangle inequality and the Cauchy--Schwarz inequality, and noting that $|\mathcal{I}_\delta| = \frac{1}\delta$, we obtain
	
	$$\|F\|_{L^p} \le \sum_{I\in\mathcal{I}_\delta} \|F_I\|_{L^p}	
	\le \left(\frac{1}{\delta}\right)^{\frac12} \left(\sum_{I\in\mathcal{I}_\delta} \|F_I\|_{L^p}^2\right)^{\frac12}.$$
	
	Define $D_p(\delta)$ to be the infimum of all constants $D$ such that
	
	$$\|F\|_{L^p(\mathbb{R}^2)} \le D \left(\sum_{I\in\mathcal{I}_\delta} \|F_I\|_{L^p(\mathbb{R}^2)}^2\right)^{\frac12}$$
	
	holds for every $F$ with $\operatorname{supp}(\widehat{F}) \subset \Omega_{[0,1]}(\delta)$. Choosing $F$ whose Fourier transform is supported on a single $\Omega_I(\delta)$ gives $D_p(\delta)\ge 1$. Together with the upper bound from the triangle inequality we have
	
	$$1 \le D_p(\delta) \le \left(\frac{1}{\delta}\right)^{\frac12}.$$
	
	An inequality of the form
	
	$$\|F\|_{L^p} \le D\left(\sum_{I\in\mathcal{I}_\delta}\|F_I\|_{L^p}^2\right)^{\frac12}$$
	
	is called a Fourier decoupling inequality. The following proposition gives bounds on $D_p(\delta)$.
	
	\begin{proposition}[Decoupling inequality for parabola, \cite{BD15}]\label{decouplingparabola}
		For every $p \in [2,\infty)$ and every $\varepsilon>0$, there exists a constant $C_{p,\varepsilon}$ depending only on $p$ and $\varepsilon$ such that for all $\delta \in 2^{-\mathbb{N}}$,
	\[
	D_p(\delta) \le
	\begin{cases}
		C_{p,\varepsilon}\,\delta^{-\varepsilon}, & 2 \le p \le 6,\\[4pt]
		C_{p,\varepsilon}\,\delta^{-\left(\frac12-\frac{3}{p}\right)-\varepsilon}, & p > 6.
	\end{cases}
	\]
	\end{proposition}
	\begin{proof}[Proof ideas (see \cite{Demeter20} for the complete proof)]
		We can use Plancherel's Theorem to prove $ D_{2}(\delta)=1 $. And we already have $ D_{\infty}(\delta)\le\delta^{-\frac{1}{2}} $. So we can finish the proof by Riesz-Thorin interpolation theorem, after we get $ D_{6}(\delta)\lesssim_{\varepsilon}\delta^{-\varepsilon} $.
	\end{proof}
	
	To generalize this decoupling inequality from the standard parabola to a more general curve, we introduce a modified frequency region associated with a function $\varphi$. Let $\varphi \in C^{2,\alpha}[0,1]$ satisfy $\varphi'' > 0$. For an interval $I \subset [0,1]$, we define the generalized frequency region as
	
	$$\Omega^\varphi_I(\delta) := \left\{\left(\xi_1, \varphi(\xi_1)+\delta'\right): \xi_1 \in I,\; |\delta'| \le \delta^2\right\}.$$
	
	Similar to the parabola case, let $F$ be an $L^p$ function whose Fourier transform is supported on $\Omega^\varphi_{[0,1]}(\delta)$, and let $F_I$ denote its Fourier restriction to $\Omega^\varphi_I(\delta)$.
	
	In order to state the localized decoupling inequality, we must introduce a spatial weight function. Given the unit ball $B=B(0,1)$ in $\mathbb{R}^2$, we denote by $\omega_B(x) = (1+\|x\|)^{-200}$ the standard weight function. For any spatial cube $Q$ centered at $x_0$ with side length $L$, the adapted weight is defined as $\omega_Q(x) = \omega_B(L^{-1}(x-x_0))$.
	
	Using our notation where $\delta$ represents the interval length in the partition $\mathcal{I}_\delta$ (which corresponds to replacing the thickness parameter $\delta$ in the standard literature with $\delta^2$), the general result from \cite{BD15} can be restated as follows:
	
	\begin{proposition}[Decoupling inequality for parabola-like curve, \cite{BD15}]\label{decouplingparabolalike}
		Let $\varphi \in C^{2,\alpha}[0,1]$ satisfy $\varphi'' > 0$. Then for every $p \in [2,\infty)$ and every $\varepsilon > 0$, there exists a constant $C_{\varphi,p,\varepsilon}$ such that
		\[
		\|F\|_{L^p(\omega_{Q_{\delta^{-2}}})} \le 
		\begin{cases}
			C_{\varphi,p,\varepsilon}\,\delta^{-\varepsilon} \left( \sum_{I\in\mathcal{I}_\delta} \|F_I\|_{L^p(\omega_{Q_{\delta^{-2}}})}^2 \right)^{\frac12}, & 2 \le p \le 6,\\
			C_{\varphi,p,\varepsilon}\,\delta^{-\left(\frac12-\frac{3}{p}\right)-\varepsilon} \left( \sum_{I\in\mathcal{I}_\delta} \|F_I\|_{L^p(\omega_{Q_{\delta^{-2}}})}^2 \right)^{\frac12}, & p > 6
		\end{cases}
		\]
		for all functions $F$ with $\operatorname{supp}(\widehat{F}) \subset \Omega^\varphi_{[0,1]}(\delta)$, where $Q_{\delta^{-2}}$ is a spatial cube of side length $\delta^{-2}$.
	\end{proposition}
	
	To extend these concepts to curves with vanishing curvature, we consider the curve $\gamma(t) = (t, t^{1+\nu})$ for a fixed $\nu > 0$. We introduce the critical parameter $r := \max\{1+\nu, 2\}$.
	
	Similar to the parabola case, for an interval $I \subset [0,1]$, we define the new frequency region associated with this curve as
	
	$$\Omega^{(\nu)}_I(\delta) := \left\{\left(\xi_1, \xi_1^{1+\nu}+\delta^{\prime}\right): \xi_1 \in I,\; \left|\delta^{\prime}\right| \leq \delta^r \right\}.$$
	
	Let $F: \mathbb{R}^2 \rightarrow \mathbb{C}$ be an $L^p$ function whose Fourier transform is supported on $\Omega^{(\nu)}_{[0,1]}(\delta)$. For each interval $I \in \mathcal{I}_\delta$, we define the Fourier restriction of $F$ to this new region by
	
	$$F_I := \mathcal{F}^{-1}\left(\widehat{F} \cdot \mathbf{1}_{\Omega^{(\nu)}_I(\delta)}\right).$$
	
	This gives the decomposition $F = \sum_{I \in \mathcal{I}_\delta} F_I$.
	
	For a rectangle $R$ in $\mathbb R^2$ with sides parallel to the coordinate axes, let $\omega_R$ denote the standard weight adapted to $R$. Since our parameter $\delta$ denotes the interval length in $\mathcal I_\delta$, the local rectangle corresponding to \cite{BGLSX20} has side lengths
	\[
	\delta^{-2}\times \delta^{-r}.
	\]
	We denote such a rectangle by $R_{\delta,r}$. We then define $D^{(\nu)}_p(\delta)$ to be the infimum of all constants $D$ such that
	
	$$\|F\|_{L^p(\omega_{R_{\delta,r}})} \leq D\left(\sum_{I \in \mathcal{I}_\delta}\left\|F_I\right\|_{L^p(\omega_{R_{\delta,r}})}^2\right)^{\frac12}$$
	
	holds for all $F$ with $\operatorname{supp}(\widehat{F}) \subset \Omega^{(\nu)}_{[0,1]}(\delta)$ and for all translates of $R_{\delta,r}$. The following proposition bounds the local decoupling constant $D^{(\nu)}_p(\delta)$ for this class of curves.
	
	\begin{proposition}[Decoupling inequality for $(t, t^{1+\nu})$, \cite{BGLSX20}]\label{decouplingmodel}
		
		For every fixed $\nu > 0$, let $r = \max\{1+\nu, 2\}$. For every $p \in [2, 6]$ and every $\varepsilon>0$, there exists a constant $C_{\nu,p, \varepsilon}$ depending on $\nu$, $p$, and $\varepsilon$ such that for all $\delta \in 2^{-\mathbb{N}}$,
		
		$$D^{(\nu)}_p(\delta) \le C_{\nu,p,\varepsilon}\,\delta^{-\varepsilon}.$$
		
	\end{proposition}
	
	By covering $\mathbb R^2$ with translates of $R_{\delta,r}$ and summing the local inequalities, one also obtains the corresponding global form
	
	$$\|F\|_{L^p(\mathbb R^2)}
	\lesssim_{\nu,p,\varepsilon}
	\delta^{-\varepsilon}
	\left(\sum_{I\in\mathcal I_\delta}
	\|F_I\|_{L^p(\mathbb R^2)}^2
	\right)^{\frac12}.$$
	
	To generalize this result to a broader class of curves, let $S$ be a regular curve in $\mathbb{R}^2$ parameterized by $\gamma(t) = (\varphi_1(t), \varphi_2(t))$ for $t \in [0,1]$, where $\varphi_1, \varphi_2 \in C^\infty([0,1])$. We assume that the Wronskian of $(\varphi_1^\prime, \varphi_2^\prime)$ vanishes only at finitely many points and to finite order. For such curves, we define the critical parameter $r$ such that $r-2$ is the maximum order of vanishing of the Wronskian.
	
	Similar to the model case, for an interval $I \subset [0,1]$, we define the frequency region associated with this general curve at parameter scale $\delta$. Since $\delta$ denotes the length of the parameter interval, the finite-type scale corresponding to order $r$ is a frequency thickness of size $\delta^r$ and a dual physical scale of size $\delta^{-r}$. We therefore use the coordinate neighborhood
	
	$$\Omega^S_I(\delta) := \left\{\xi \in \mathbb{R}^2 : \exists t \in I \text{ such that } |\xi_1-\varphi_1(t)|\le \delta^r,\ |\xi_2-\varphi_2(t)|\le \delta^r \right\}.$$
	
	For $r=2$, this recovers the usual $\delta^2$-neighborhood used for non-degenerate curves.
	
	Let $F: \mathbb{R}^2 \rightarrow \mathbb{C}$ be an $L^p$ function whose Fourier transform is supported on $\Omega^S_{[0,1]}(\delta)$. For each interval $I \in \mathcal{I}_\delta$, we define the Fourier restriction of $F$ to this new region by
	
	$$F_I := \mathcal{F}^{-1}\left(\widehat{F} \cdot \mathbf{1}_{\Omega^S_I(\delta)}\right).$$
	
	This gives the decomposition $F = \sum_{I \in \mathcal{I}_\delta} F_I$. Let $B_{\delta,r}$ denote a spatial ball of radius $\delta^{-r}$. We then define $D^S_p(\delta)$ to be the infimum of all constants $D$ such that
	
	$$\|F\|_{L^p(\omega_{B_{\delta,r}})} \leq D\left(\sum_{I \in \mathcal{I}_\delta}\left\|F_I\right\|_{L^p\left(\omega_{B_{\delta,r}}\right)}^2\right)^{\frac12}$$
	
	holds for all $F$ with $\operatorname{supp}(\widehat{F}) \subset \Omega^S_{[0,1]}(\delta)$ and for all translates of $B_{\delta,r}$. The following proposition bounds the local weighted decoupling constant $D^S_p(\delta)$ for this general class of curves.
	
	\begin{proposition}[Decoupling inequality for general regular curves, \cite{BGLSX20}]\label{decouplingregular}
		
		Let $S$ be a regular curve as described above, and let $r-2$ be the maximum order of vanishing of the Wronskian of $(\varphi_1^\prime, \varphi_2^\prime)$. For every $p \in [2, 6]$ and every $\varepsilon>0$, there exists a constant $C_{S,p, \varepsilon}$ depending on the curve, $p$, and $\varepsilon$ such that for all $\delta \in 2^{-\mathbb{N}}$,
		
		$$D^S_p(\delta) \le C_{S,p,\varepsilon}\,\delta^{-\varepsilon}.$$
		
	\end{proposition}
	
	Again, by covering $\mathbb R^2$ with translates of $B_{\delta,r}$ and summing, this local weighted estimate implies the global form
	
	$$\|F\|_{L^p(\mathbb R^2)}
	\lesssim_{S,p,\varepsilon}
	\delta^{-\varepsilon}
	\left(\sum_{I\in\mathcal I_\delta}
	\|F_I\|_{L^p(\mathbb R^2)}^2
	\right)^{\frac12}.$$
	
	\section{Proof of Theorem~\ref{singlefraction}}\label{proofsingle}
	We will use Proposition~\ref{decouplingmodel} to prove Theorem~\ref{singlefraction}.
	
	\begin{proof}[Proof of Theorem~\ref{singlefraction}]
		Let
		\[
		S_N(x,t):=\sum_{n=1}^N a_n e^{2\pi i(nx+n^{1+\nu}t)}
		\]
		and
		\[
		S_N^\circ(x,t):=\sum_{n=1}^{N-1} a_n e^{2\pi i(nx+n^{1+\nu}t)}.
		\]
		The endpoint term $a_Ne^{2\pi i(Nx+N^{1+\nu}t)}$ has $L^p([0,1]^2)$ norm equal to $|a_N|$, which is bounded by $\|a\|_{\ell^2}$. Since the endpoint exponent in the theorem is non-negative, it suffices to prove the desired estimate for $S_N^\circ$.
		
		By the Cauchy-Schwarz inequality,
		\[
		\|S_N\|_{L^\infty([0,1]^2)}
		\le N^{\frac12}\left(\sum_{n=1}^N |a_n|^2\right)^{\frac12}.
		\]
		Moreover, since the spatial frequencies $n$ are distinct integers, orthogonality in the $x$ variable gives
		\[
		\|S_N\|_{L^2([0,1]^2)}
		=
		\left(\sum_{n=1}^N |a_n|^2\right)^{\frac12}.
		\]
		Therefore it suffices to prove the $L^6$ endpoint estimate
		\[
		\|S_N^\circ\|_{L^6([0,1]^2)}
		\lesssim_{\nu,\varepsilon}
		N^{\frac{r-1-\nu}{6}+\varepsilon}
		\left(\sum_{n=1}^N |a_n|^2\right)^{\frac12},
		\qquad r=\max\{2,1+\nu\}.
		\]
		The full statement then follows by interpolation with the $L^2$ and $L^\infty$ estimates above.
		
		Set $\delta=N^{-1}$. If $\delta$ is not dyadic, we replace it by a comparable dyadic number; this changes only the implicit constants. Choose a smooth compactly supported cutoff $\chi_0:\mathbb R^2\to\mathbb C$ whose support is contained in a sufficiently small ball around the origin, and whose inverse Fourier transform $\check{\chi}_0$ is real and satisfies
		\[
		|\check{\chi}_0(y_1,y_2)|\ge c_0>0,
		\qquad (y_1,y_2)\in[0,1]^2.
		\]
		Define $F:\mathbb R^2\to\mathbb C$ by prescribing its Fourier transform:
		\[
		\widehat F(\xi_1,\xi_2)
		=
		\sum_{n=1}^{N-1} a_n
		\chi_0\left(
		N^r\left(\xi_1-\frac nN\right),
		N^r\left(\xi_2-\left(\frac nN\right)^{1+\nu}\right)
		\right).
		\]
		Since $N^{-r}\le N^{-1}=\delta$, and since the vertical error produced by the horizontal thickness is also $O(N^{-r})$, the support of $\widehat F$ is contained in a fixed enlargement of $\Omega^{(\nu)}_{[0,1]}(\delta)$. This fixed enlargement affects only the implicit constants. Let $F_n$ denote the corresponding summand.
		
		Let $R_{N^{-1},r}$ be a rectangle with side lengths
		\[
		N^2\times N^r
		\]
		whose center is the midpoint of
		\[
		D_\nu=[0,N^2]\times[0,N^{1+\nu}].
		\]
		Applying the local weighted form of Proposition~\ref{decouplingmodel} at $p=6$ gives
		\[
		\|F\|_{L^6(\omega_{R_{N^{-1},r}})}
		\lesssim_{\nu,\varepsilon}
		N^\varepsilon
		\left(\sum_{n=1}^{N-1} \|F_n\|_{L^6(\omega_{R_{N^{-1},r}})}^2\right)^{\frac12}.
		\]
		By the inverse Fourier transform,
		\[
		F_n(y_1,y_2)
		=
		a_n N^{-2r}
		e^{2\pi i\left(y_1\frac nN+y_2\left(\frac nN\right)^{1+\nu}\right)}
		\check{\chi}_0\left(\frac{y_1}{N^r},\frac{y_2}{N^r}\right).
		\]
		Since the weight is adapted to a rectangle of side lengths $N^2\times N^r$, we have
		\[
		\|F_n\|_{L^6(\omega_{R_{N^{-1},r}})}
		\lesssim
		|a_n|N^{-2r}(N^2N^r)^{\frac16}
		=
		|a_n|N^{-2r+\frac{r+2}{6}}.
		\]
		Consequently,
		\[
		\|F\|_{L^6(\omega_{R_{N^{-1},r}})}
		\lesssim_{\nu,\varepsilon}
		N^{-2r+\frac{r+2}{6}+\varepsilon}
		\left(\sum_{n=1}^{N-1} |a_n|^2\right)^{\frac12}.
		\]
		
		It remains to extract the desired discrete sum from $F$. Since $r\ge2$ and $r\ge1+\nu$, we have
		\[
		\frac{y_1}{N^r}\in[0,1],
		\qquad
		\frac{y_2}{N^r}\in[0,1]
		\]
		throughout $D_\nu=[0,N^2]\times[0,N^{1+\nu}]$. Therefore the cutoff factor is bounded below on $D_\nu$. Moreover, by the choice of $R_{N^{-1},r}$, we have $\omega_{R_{N^{-1},r}}\gtrsim1$ on $D_\nu$. Thus
		\[
		\|F\|_{L^6(\omega_{R_{N^{-1},r}})}
		\gtrsim
		N^{-2r}
		\left(
		\int_0^{N^2}\int_0^{N^{1+\nu}}
		\left|
		\sum_{n=1}^{N-1} a_n
		e^{2\pi i\left(y_1\frac nN+y_2\frac{n^{1+\nu}}{N^{1+\nu}}\right)}
		\right|^6dy_2dy_1
		\right)^{\frac16}.
		\]
		Now make the change of variables $y_1=Nx$ and $y_2=N^{1+\nu}t$. Then $x\in[0,N]$ and $t\in[0,1]$. The integrand is $1$-periodic in $x$, because the spatial frequencies are the integers $n$. Hence the $x$ integration contains exactly $N$ copies of the unit interval, and we obtain
		\[
		\|F\|_{L^6(\omega_{R_{N^{-1},r}})}
		\gtrsim
		N^{-2r+\frac{3+\nu}{6}}
		\|S_N^\circ\|_{L^6([0,1]^2)}.
		\]
		Combining the lower and upper bounds gives
		\[
		\|S_N^\circ\|_{L^6([0,1]^2)}
		\lesssim_{\nu,\varepsilon}
		N^{\frac{r-1-\nu}{6}+\varepsilon}
		\left(\sum_{n=1}^{N} |a_n|^2\right)^{\frac12}.
		\]
		Adding back the endpoint term $n=N$ gives the same estimate for the full sum $S_N$.
		
		For $2\le p\le6$, interpolation between $L^2$ and $L^6$ gives the exponent
		\[
		\frac{(r-1-\nu)(p-2)}{4p}.
		\]
		For $p>6$, interpolation between $L^6$ and $L^\infty$ gives the exponent
		\[
		\frac12+\frac{r-4-\nu}{p}.
		\]
		This is exactly the stated value of $E_p$, completing the proof.
	\end{proof}
	\section{Proof of Theorem~\ref{doublefraction}}\label{proofdouble}
	Due to a minor difference in the exponents between Theorem~\ref{doublefraction} and Theorem~\ref{singlefraction}, Proposition~\ref{decouplingmodel} is no longer applicable. Fortunately, however, we can adapt the method employed in \cite{BGLSX20} to prove a similar proposition.
	
	To generalize this decoupling inequality to curves of the form $\Gamma(t) = (t^{1+\mu}, t^{1+\nu})$, we first define the associated anisotropic frequency region. Set
	\[
	r_1=\max\{2,1+\mu\},\qquad r_2=\max\{2,1+\nu\}.
	\]
	For any interval $I \subset [0,1]$, we define
	
	$$\Omega^\Gamma_I(\delta) := \left\{ (t^{1+\mu} + \delta_1, t^{1+\nu} + \delta_2) : t \in I, \; |\delta_1| \le \delta^{r_1},\ |\delta_2| \le \delta^{r_2} \right\}.$$
	
	Given the unit ball $B=B(0,1)$ in $\mathbb{R}^2$, we define the standard weight function $\omega_B(x) = (1+\|x\|)^{-200}$. For any rectangle $R$ centered at $x_0$ with side lengths $L_1$ and $L_2$ parallel to the coordinate axes, the adapted weight is defined as $\omega_R(x) = \omega_B(T_R^{-1}(x-x_0))$, where $T_R$ is the diagonal linear transformation mapping the standard basis vectors to lengths $L_1$ and $L_2$, respectively.
	
	\begin{proposition}\label{decouplingdouble}
		Let $\Gamma(t) = (t^{1+\mu}, t^{1+\nu})$ with $0 < \mu < \nu$. Let $\delta \in (0,1)$ be the partition scale, and let $\mathcal{I}_\delta$ denote the partition of $[0,1]$ into intervals of length $\delta$. For all $2 \le p \le 6$ and all $\varepsilon > 0$, there exists a constant $C_{\mu,\nu,p,\varepsilon} > 0$ such that
		\[
		\|F\|_{L^p(\omega_{R_\delta})}
		\le
		C_{\mu,\nu,p,\varepsilon}\delta^{-\varepsilon}
		\left(
		\sum_{J\in\mathcal I_\delta}
		\|F_J\|_{L^p(\omega_{R_\delta})}^2
		\right)^{\frac12}
		\]
		holds for all functions $F$ with $\operatorname{supp}(\widehat F)\subset\Omega^\Gamma_{[0,1]}(\delta)$, where $F_J$ denotes the Fourier restriction of $F$ to $\Omega^\Gamma_J(\delta)$. Here $R_\delta$ is a global spatial rectangle in $\mathbb R^2$ with side lengths $\delta^{-r_1}$ and $\delta^{-r_2}$.
	\end{proposition}
	
	\begin{proof}
		To isolate the degeneracy of $\Gamma$ at the origin, fix an auxiliary parameter $0<\eta\ll\varepsilon$ and dyadically decompose
		\[
		(0,1]=(0,\delta^{1-\eta}]\cup\bigcup_a [a,2a],
		\qquad \delta^{1-\eta}\le a\le \frac12.
		\]
		The core interval $(0,\delta^{1-\eta}]$ contains $O(\delta^{-\eta})$ intervals of length $\delta$, so the trivial $\ell^1$ estimate followed by Cauchy-Schwarz gives a loss $O(\delta^{-\eta/2})$. Since $\eta$ is chosen sufficiently small compared with the final loss exponent $\varepsilon$, this is absorbed into $\delta^{-\varepsilon}$.
		
		We now consider a dyadic block $I_a=[a,2a]$. Write $t=as$, where $s\in[1,2]$. The effective parameter scale becomes
		\[
		\widetilde{\delta}=\frac{\delta}{a}.
		\]
		Under the physical rescaling
		\[
		\widetilde{x}_1=a^{1+\mu}x_1,\qquad
		\widetilde{x}_2=a^{1+\nu}x_2,
		\]
		the curve becomes
		\[
		\widetilde{\Gamma}(s)=(s^{1+\mu},s^{1+\nu}),\qquad s\in[1,2].
		\]
		This curve is smooth and non-degenerate on $[1,2]$, so Proposition~\ref{decouplingparabolalike} applies at scale $\widetilde{\delta}$.
		
		The anisotropic frequency widths after this rescaling are
		\[
		W_1=\delta^{r_1}a^{-(1+\mu)},\qquad
		W_2=\delta^{r_2}a^{-(1+\nu)}.
		\]
		We verify that these widths are bounded by the regular normal thickness $\widetilde{\delta}^2=\delta^2a^{-2}$. For the first coordinate, the required inequality is
		\[
		\delta^{r_1}a^{-(1+\mu)}\le \delta^2a^{-2}.
		\]
		If $\mu\le1$, then $r_1=2$, and the inequality follows from $a^{-(1+\mu)}\le a^{-2}$. If $\mu>1$, then $r_1=1+\mu$, and the inequality becomes $(\delta/a)^{\mu-1}\le1$, which follows from $a\ge\delta^{1-\eta}\ge\delta$. The second coordinate is identical, with $\nu$ replacing $\mu$. Hence the rescaled support lies inside the standard $\widetilde{\delta}^2$-neighborhood of a non-degenerate curve.
		
		Applying the regular decoupling theorem with loss exponent $\varepsilon/10$ gives a local estimate on rectangles whose side lengths, after returning to the original physical variables, are
		\[
		L_1(a)=\widetilde{\delta}^{-2}a^{-(1+\mu)}
		=
		a^{1-\mu}\delta^{-2},
		\qquad
		L_2(a)=\widetilde{\delta}^{-2}a^{-(1+\nu)}
		=
		a^{1-\nu}\delta^{-2}.
		\]
		From the same elementary case distinction as above, we have
		\[
		L_1(a)\le \delta^{-r_1},\qquad L_2(a)\le \delta^{-r_2}.
		\]
		Thus every local rectangle is contained, up to constants, in the global rectangle $R_\delta$.
		
		It remains to pass from the local rectangles to the global weighted rectangle. Cover $R_\delta$ by bounded-overlap translates of the local rectangle $R_a$. Applying the local decoupling estimate on each translate and then summing, the vector-valued Minkowski inequality, using $p/2\ge1$, gives
		\[
		\|F_{I_a}\|_{L^p(\omega_{R_\delta})}
		\lesssim
		\widetilde{\delta}^{-\varepsilon/10}
		\left(
		\sum_{J\subset I_a}
		\|F_J\|_{L^p(\omega_{R_\delta})}^2
		\right)^{\frac12}.
		\]
		Since $\widetilde{\delta}^{-\varepsilon/10}=(\delta/a)^{-\varepsilon/10}\le \delta^{-\varepsilon/10}$, this is sufficient on each dyadic block. Finally, summing over the $O(\log(\delta^{-1}))$ dyadic blocks by the triangle inequality and Cauchy-Schwarz introduces only a logarithmic loss, which is absorbed into $\delta^{-\varepsilon}$. This completes the proof.
	\end{proof}
	
	Now we can start the proof of Theorem~\ref{doublefraction}.
	
	\begin{proof}[Proof of Theorem~\ref{doublefraction}]
		Let
		\[
		S_N(x,t):=
		\sum_{n=1}^N a_n e^{2\pi i(x n^{1+\mu}+t n^{1+\nu})}
		\]
		and
		\[
		S_N^\circ(x,t):=
		\sum_{n=1}^{N-1} a_n e^{2\pi i(x n^{1+\mu}+t n^{1+\nu})}.
		\]
		The endpoint term $n=N$ has $L^p([0,1]^2)$ norm equal to $|a_N|$, and is therefore harmless. It suffices to prove the desired estimate for $S_N^\circ$.
		
		The $L^\infty$ bound is immediate:
		\[
		\|S_N\|_{L^\infty([0,1]^2)}
		\le
		N^{\frac12}
		\left(\sum_{n=1}^N |a_n|^2\right)^{\frac12}.
		\]
		For the $L^2$ bound, we use the fact that the frequencies $\{n^{1+\mu}\}_{n\ge1}$ are uniformly separated. By the standard Bessel inequality for separated exponentials on $[0,1]$, uniformly in $t$,
		\[
		\int_0^1\left|
		\sum_{n=1}^N a_n e^{2\pi i t n^{1+\nu}}e^{2\pi i x n^{1+\mu}}
		\right|^2dx
		\lesssim_{\mu}
		\sum_{n=1}^N |a_n|^2.
		\]
		Integrating in $t$ gives
		\[
		\|S_N\|_{L^2([0,1]^2)}
		\lesssim_{\mu}
		\left(\sum_{n=1}^N |a_n|^2\right)^{\frac12}.
		\]
		Therefore it remains to prove the $L^6$ endpoint estimate
		\[
		\|S_N^\circ\|_{L^6([0,1]^2)}
		\lesssim_{\mu,\nu,\varepsilon}
		N^{\frac{r_1+r_2-2-\mu-\nu}{6}+\varepsilon}
		\left(\sum_{n=1}^N |a_n|^2\right)^{\frac12},
		\]
		where
		\[
		r_1=\max\{2,1+\mu\},\qquad r_2=\max\{2,1+\nu\}.
		\]
		
		Set $\delta=N^{-1}$. If $\delta$ is not dyadic, we replace it by a comparable dyadic number; this changes only the implicit constants. Choose a smooth compactly supported cutoff $\chi_0:\mathbb R^2\to\mathbb C$ whose support is contained in a sufficiently small ball around the origin, and whose inverse Fourier transform $\check{\chi}_0$ is real and satisfies
		\[
		|\check{\chi}_0(y_1,y_2)|\ge c_0>0,
		\qquad (y_1,y_2)\in[0,1]^2.
		\]
		Define $F=\sum_{n=1}^{N-1} F_n$ through its Fourier transform:
		\[
		\widehat{F}(\xi_1, \xi_2)
		=
		\sum_{n=1}^{N-1} a_n
		\chi_0\left(
		N^{r_1}\left(\xi_1-\left(\frac nN\right)^{1+\mu}\right),
		N^{r_2}\left(\xi_2-\left(\frac nN\right)^{1+\nu}\right)
		\right).
		\]
		By construction, $\operatorname{supp}(\widehat F)\subset\Omega^\Gamma_{[0,1]}(\delta)$, up to a harmless fixed enlargement of the frequency neighborhoods. Choose $R_\delta$ to be a coordinate rectangle with side lengths $N^{r_1}$ and $N^{r_2}$ whose center is the midpoint of
		\[
		D=[0,N^{1+\mu}]\times[0,N^{1+\nu}].
		\]
		Proposition~\ref{decouplingdouble} gives
		\[
		\|F\|_{L^6(\omega_{R_\delta})}
		\lesssim_{\mu,\nu,\varepsilon}
		N^\varepsilon
		\left(
		\sum_{n=1}^{N-1}
		\|F_n\|_{L^6(\omega_{R_\delta})}^2
		\right)^{\frac12}.
		\]
		The inverse Fourier transform gives
		\[
		F_n(y_1,y_2)
		=
		a_n N^{-r_1-r_2}
		e^{2\pi i\left(y_1\left(\frac nN\right)^{1+\mu}
			+y_2\left(\frac nN\right)^{1+\nu}\right)}
		\check{\chi}_0\left(\frac{y_1}{N^{r_1}},\frac{y_2}{N^{r_2}}\right).
		\]
		Using $\omega_{R_\delta}\le1$ and the rapid decay of $\check{\chi}_0$, we obtain
		\[
		\|F_n\|_{L^6(\omega_{R_\delta})}
		\le
		\|F_n\|_{L^6(\mathbb R^2)}
		\lesssim
		|a_n|N^{-r_1-r_2}N^{\frac{r_1+r_2}{6}}
		=
		|a_n|N^{-\frac56(r_1+r_2)}.
		\]
		Therefore
		\[
		\|F\|_{L^6(\omega_{R_\delta})}
		\lesssim_{\mu,\nu,\varepsilon}
		N^\varepsilon
		N^{-\frac56(r_1+r_2)}
		\left(\sum_{n=1}^{N-1} |a_n|^2\right)^{\frac12}.
		\]
		
		We now obtain the lower bound. Since $r_1\ge1+\mu$ and $r_2\ge1+\nu$, the rectangle $D$ is contained in a fixed dilate of $R_\delta$, and therefore $\omega_{R_\delta}\gtrsim1$ on $D$. The cutoff factor is also bounded below on $D$. Hence
		\[
		\|F\|_{L^6(\omega_{R_\delta})}
		\gtrsim
		\|F\|_{L^6(D)}
		\gtrsim
		N^{-r_1-r_2}
		N^{\frac{2+\mu+\nu}{6}}
		\|S_N^\circ\|_{L^6([0,1]^2)},
		\]
		after the change of variables $y_1=N^{1+\mu}x$ and $y_2=N^{1+\nu}t$.
		
		Combining the upper and lower bounds gives
		\[
		\|S_N^\circ\|_{L^6([0,1]^2)}
		\lesssim_{\mu,\nu,\varepsilon}
		N^{\frac{r_1+r_2-2-\mu-\nu}{6}+\varepsilon}
		\left(\sum_{n=1}^N |a_n|^2\right)^{\frac12}.
		\]
		Adding back the endpoint term $n=N$ gives the full estimate for $S_N$.
		
		Finally, interpolation with the $L^2$ estimate gives, for $2\le p\le6$,
		\[
		E_p=\frac{(p-2)(r_1+r_2-2-\mu-\nu)}{4p}.
		\]
		Interpolation with the $L^\infty$ estimate gives, for $p>6$,
		\[
		E_p=\frac12+\frac{r_1+r_2-5-\mu-\nu}{p}.
		\]
		This completes the proof.
	\end{proof}
	\section{Proof of Theorem~\ref{perturbation}}\label{proofperturbation}
	Theorem~\ref{perturbation} is, in effect, an extension of the discrete Strichartz estimate corresponding to the model curve $(t,t^k)$; we will use Proposition~\ref{decouplingmodel} to prove the required endpoint estimate.
	
	\begin{proof}[Proof of Theorem~\ref{perturbation}]
		Let
		\[
		S_N(x,t)
		=
		\sum_{n=0}^N a_n e^{2\pi i(nx+(n^k+\theta_n)t)}
		\]
		and
		\[
		S_N^\circ(x,t)
		=
		\sum_{n=1}^{N-1} a_n e^{2\pi i(nx+(n^k+\theta_n)t)}.
		\]
		The endpoint terms $n=0$ and $n=N$ have $L^p([0,1]^2)$ norms bounded by $|a_0|+|a_N|$, and are therefore harmless. It suffices to prove the endpoint $L^6$ estimate for $S_N^\circ$.
		
		Orthogonality in the spatial variable gives
		\[
		\|S_N\|_{L^2([0,1]^2)}
		=
		\left(\sum_{n=0}^N |a_n|^2\right)^{\frac12},
		\]
		and the Cauchy-Schwarz inequality gives
		\[
		\|S_N\|_{L^\infty([0,1]^2)}
		\le
		N^{\frac12}
		\left(\sum_{n=0}^N |a_n|^2\right)^{\frac12}.
		\]
		Thus it suffices to prove
		\[
		\|S_N^\circ\|_{L^6([0,1]^2)}
		\lesssim_{k,\varepsilon,C}
		N^\varepsilon
		\left(\sum_{n=0}^N |a_n|^2\right)^{\frac12}.
		\]
		
		Set $\delta=N^{-1}$ and $A=N^k$. If $\delta$ is not dyadic, we replace it by a comparable dyadic number; this changes only the implicit constants. Choose a smooth compactly supported cutoff $\chi_0:\mathbb R^2\to\mathbb C$ whose support is contained in a sufficiently small ball around the origin, and whose inverse Fourier transform $\check{\chi}_0$ is real and satisfies
		\[
		|\check{\chi}_0(y_1,y_2)|\ge c_0>0,
		\qquad (y_1,y_2)\in[0,1]^2.
		\]
		Define $F=\sum_{n=1}^{N-1} F_n$ by
		\[
		\widehat F(\xi_1,\xi_2)
		=
		\sum_{n=1}^{N-1} a_n
		\chi_0\left(
		A\left(\xi_1-\frac nN\right),
		A\left(\xi_2-\frac{n^k+\theta_n}{N^k}\right)
		\right).
		\]
		The center of the $n$-th frequency bump is
		\[
		\left(\frac nN,\frac{n^k+\theta_n}{N^k}\right).
		\]
		Its vertical distance from the unperturbed curve $(s,s^k)$ at $s=\frac nN$ is bounded by
		\[
		\frac{|\theta_n|}{N^k}\le C N^{-k}=C\delta^k.
		\]
		The horizontal width of each bump is $A^{-1}=N^{-k}\le N^{-1}$, and the corresponding vertical error caused by the curvature of $(s,s^k)$ is also $O(N^{-k})$. Hence the Fourier support of $F$ is contained in a fixed $C$-dependent enlargement of the $\delta^k$-neighborhood of the curve $(s,s^k)$.
		
		Let $R_{N^{-1},k}$ be a rectangle with side lengths
		\[
		N^2\times N^k
		\]
		whose center is the midpoint of
		\[
		D_k=[0,N^2]\times[0,N^k].
		\]
		Applying the local weighted form of Proposition~\ref{decouplingmodel}, with $\nu=k-1$, gives
		\[
		\|F\|_{L^6(\omega_{R_{N^{-1},k}})}
		\lesssim_{k,\varepsilon,C}
		N^\varepsilon
		\left(\sum_{n=1}^{N-1}\|F_n\|_{L^6(\omega_{R_{N^{-1},k}})}^2\right)^{\frac12}.
		\]
		By the inverse Fourier transform,
		\[
		F_n(y_1,y_2)
		=
		a_n A^{-2}
		e^{2\pi i\left(y_1\frac nN+y_2\frac{n^k+\theta_n}{N^k}\right)}
		\check{\chi}_0\left(\frac{y_1}{A},\frac{y_2}{A}\right).
		\]
		Since $A=N^k$ and the weight is adapted to a rectangle of side lengths $N^2\times N^k$, we have
		\[
		\|F_n\|_{L^6(\omega_{R_{N^{-1},k}})}
		\lesssim
		|a_n|A^{-2}(N^2N^k)^{\frac16}
		=
		|a_n|N^{-2k+\frac{k+2}{6}}.
		\]
		Consequently,
		\[
		\|F\|_{L^6(\omega_{R_{N^{-1},k}})}
		\lesssim_{k,\varepsilon,C}
		N^\varepsilon N^{-2k+\frac{k+2}{6}}
		\left(\sum_{n=1}^{N-1} |a_n|^2\right)^{\frac12}.
		\]
		
		For the lower bound, restrict the physical integral to $D_k=[0,N^2]\times[0,N^k]$. Since $k\ge2$, we have $y_1/A\in[0,1]$ and $y_2/A\in[0,1]$ on $D_k$, and therefore the cutoff factor is bounded below on $D_k$. Moreover, by the choice of $R_{N^{-1},k}$, we have $\omega_{R_{N^{-1},k}}\gtrsim1$ on $D_k$. Hence
		\[
		\|F\|_{L^6(\omega_{R_{N^{-1},k}})}
		\gtrsim
		A^{-2}
		\left(
		\int_0^{N^2}\int_0^{N^k}
		\left|
		\sum_{n=1}^{N-1} a_n
		e^{2\pi i\left(y_1\frac nN+y_2\frac{n^k+\theta_n}{N^k}\right)}
		\right|^6dy_2dy_1
		\right)^{\frac16}.
		\]
		Now make the change of variables $y_1=Nx$ and $y_2=N^k t$. Then $x\in[0,N]$ and $t\in[0,1]$. The integrand is $1$-periodic in $x$, so the $x$ integration contains exactly $N$ copies of the unit interval. Hence
		\[
		\|F\|_{L^6(\omega_{R_{N^{-1},k}})}
		\gtrsim
		N^{-2k}N^{\frac{k+2}{6}}
		\|S_N^\circ\|_{L^6([0,1]^2)}.
		\]
		Combining the upper and lower bounds yields
		\[
		\|S_N^\circ\|_{L^6([0,1]^2)}
		\lesssim_{k,\varepsilon,C}
		N^\varepsilon
		\left(\sum_{n=0}^N |a_n|^2\right)^{\frac12}.
		\]
		Adding back the two endpoint terms $n=0$ and $n=N$ gives the stated estimate for the full sum. Interpolation with the $L^2$ and $L^\infty$ bounds gives the stated estimate for every $p\ge2$.
	\end{proof}
	\subsection{Two Applications}
	We conclude with a concrete application of Theorem~\ref{perturbation}. Let
	\[
	D_x=\frac{1}{2\pi i}\partial_x
	\]
	so that
	\[
	D_x e^{2\pi i n x}=n e^{2\pi i n x}.
	\]
	Fix an integer $k\ge2$ and a real parameter $\lambda$. Consider the linear dispersive equation
	\[
	i\partial_t u+2\pi\left(D_x^k+\lambda\sin(D_x^2)\right)u=0,
	\qquad (x,t)\in \mathbb T\times[0,1].
	\]
	Here the perturbation is nonlinear as a function of the frequency operator, although the equation itself remains linear. The operator $\sin(D_x^2)$ is a bounded Fourier multiplier, since
	\[
	\sin(D_x^2)e^{2\pi i n x}=\sin(n^2)e^{2\pi i n x}.
	\]
	For initial data
	\[
	u_0(x)=\sum_{n=0}^N a_n e^{2\pi i n x},
	\]
	the corresponding solution is
	\[
	u(x,t)
	=
	\sum_{n=0}^N
	a_n
	e^{2\pi i\left(nx+\left(n^k+\lambda\sin(n^2)\right)t\right)}.
	\]
	This is exactly the setting of Theorem~\ref{perturbation}, with
	\[
	\theta_n=\lambda\sin(n^2),
	\qquad
	|\theta_n|\le |\lambda|.
	\]
	Therefore, for every $p\ge2$ and every $\varepsilon>0$,
	\[
	\|u\|_{L^p(\mathbb T\times[0,1])}
	\lesssim_{k,p,\varepsilon,\lambda}
	N^\varepsilon
	\left(1+N^{\frac12-\frac3p}\right)
	\|u_0\|_{L^2(\mathbb T)}.
	\]
	
	This example illustrates the advantage of the perturbative decoupling approach. After normalization, the perturbed frequency points satisfy
	\[
	\left(
	\frac nN,
	\frac{n^k+\lambda\sin(n^2)}{N^k}
	\right)
	=
	\left(
	\frac nN,
	\left(\frac nN\right)^k
	\right)
	+
	\left(0,O(N^{-k})\right),
	\]
	so they remain inside the natural $N^{-k}$-neighborhood of the model curve
	\[
	(s,s^k).
	\]
	Thus the bounded spectral perturbation does not change the decoupling scale.
	
	On the other hand, this estimate is not easily accessible through purely arithmetic methods such as efficient congruencing. Expanding the sixth moment would introduce phase interactions involving terms such as
	\[
	\sin(n_1^2)+\sin(n_2^2)+\sin(n_3^2)
	-
	\sin(m_1^2)-\sin(m_2^2)-\sin(m_3^2),
	\]
	which do not lead to a polynomial Diophantine system. The decoupling argument avoids this obstruction by using only the geometric fact that the perturbed points remain inside the same frequency tube.
	
	Similarly, one may replace $\lambda\sin(D_x^2)$ by much more irregular bounded Fourier multipliers. For instance, let $\Pi_{\mathbb P}$ be the prime-frequency projection defined by
	\[
	\Pi_{\mathbb P}e^{2\pi i n x}
	=
	\mathbf 1_{\mathbb P}(n)e^{2\pi i n x}.
	\]
	Then the equation
	\[
	i\partial_t u+2\pi\left(D_x^k+\lambda\Pi_{\mathbb P}\right)u=0
	\]
	has frequency-localized solutions of the form
	\[
	u(x,t)
	=
	\sum_{n=0}^N
	a_n
	e^{2\pi i\left(nx+\left(n^k+\lambda\mathbf 1_{\mathbb P}(n)\right)t\right)}.
	\]
	Since $\lambda\mathbf 1_{\mathbb P}(n)$ is bounded, Theorem~\ref{perturbation} again gives
	\[
	\|u\|_{L^p(\mathbb T\times[0,1])}
	\lesssim_{k,p,\varepsilon,\lambda}
	N^\varepsilon
	\left(1+N^{\frac12-\frac3p}\right)
	\|u_0\|_{L^2(\mathbb T)}.
	\]
	This provides a discrete Strichartz estimate for dispersive flows whose lower-order spectral perturbations have no polynomial structure.
	\section{Final remarks}\label{finalremarks}
	\subsection{The polynomial comparison}\label{polynomialproof}
	We briefly recall the decoupling mechanism behind Remark~\ref{polynomial}.  The continuous input is the following standard finite-type consequence of curve decoupling.  Let $\gamma(t)=(\varphi_1(t),\varphi_2(t))$ be a polynomial curve on $[0,1]$ such that $\varphi_1'$ and $\varphi_2'$ are linearly independent, and let
	$$
	W(t)=\varphi_1'(t)\varphi_2''(t)-\varphi_2'(t)\varphi_1''(t)
	$$
	be its Wronskian.  Since $W$ is a non-zero polynomial, it has only finitely many zeros.  After decomposing $[0,1]$ into finitely many intervals and using the finite-type decoupling theorem near the possible zeros of $W$, one obtains, for $2\le p\le6$,
	$$
	D_p^{(\gamma,r)}(\delta)\lesssim_{\gamma,p,\varepsilon}\delta^{-\varepsilon},
	$$
	where $r$ is a finite type parameter depending only on the polynomial curve.
	
	To pass from this continuous estimate to the discrete polynomial estimate, one applies the same localization argument used throughout the paper.  If $P$ and $Q$ have the same degree, a fixed rational linear change of variables first reduces the pair to one with distinct leading degrees.  On each dyadic block $n\sim M$, the normalized curve
	$$
	u\mapsto \left(\frac{P(Mu)}{M^{k_1}},\frac{Q(Mu)}{M^{k_2}}\right)
	$$
	has uniformly controlled finite type on $u\in[1/2,1]$, so the preceding decoupling estimate gives the $L^6$ bound with only an $M^\varepsilon$ loss.  Summing the dyadic blocks gives the $p=6$ endpoint.  The full range $p\ge2$ then follows by interpolation with the elementary $L^2$ and $L^\infty$ estimates.  This explains why the polynomial estimate has the decoupling exponent $\left(\frac{1}2-\frac{3}p\right)_+$, although it is not treated as one of the main new results of the paper.
	
	\subsection{Some comments on the results}
	The conjectured bounds presented in the \hyperref[table]{table} can be derived using the circle method or other heuristic approaches; we will not elaborate on these details here. By examining the proofs of these theorems and specifically noting the comparison between the bounds obtained via the decoupling method and the efficient congruencing method---we can clearly discern the respective strengths of each approach.
	
	In summary, the decoupling method relies primarily on decoupling inequalities and is independent of the number-theoretic significance of the expressions involved; consequently, it remains highly effective even in cases involving fractional powers. (In such instances, the efficient congruencing method proves entirely ineffectual---yielding only trivial bounds---as it fundamentally relies on the congruence properties of integers.)
	
	However, the decoupling method also suffers from a rather evident drawback: specifically (in the two-dimensional setting), it yields sharp bounds only for cases where $p \le 6$. Yet, in certain scenarios, such as Conjecture~\ref{kdvconjecture}, we aim to establish sharp bounds for cases extending up to $p \le 8$; in these situations, the efficient congruencing method proves to be the more effective tool.
	\section*{Acknowledgment}
	The author wishes to thank Prof. Daniel Tataru and Prof. Ruixiang Zhang for their suggestions, discussions, and continued support during the development of this paper. The author also thanks Prof. Shaoming Guo, Dr. Mingfeng Chen, and Dr. Minxing Shen for providing research background and reference materials during the preliminary stages of the study.
	\bibliographystyle{plain}
	\bibliography{references}
\end{document}